\documentclass[preprint,12pt]{elsarticle}

\usepackage{graphicx}
\usepackage{amssymb,stmaryrd}
\usepackage{amsthm}
\usepackage{amsmath}
\usepackage{color}
\usepackage{comment}
\usepackage{calligra}
\DeclareMathAlphabet{\mathcalligra}{T1}{calligra}{m}{n}
\DeclareFontShape{T1}{calligra}{m}{n}{<->s*[2.2]callig15}{}

\journal{Discrete Mathematics}

\newcommand{\ml}{l\kern-0.55mm\char39\kern-0.3mm}

\newtheorem{theorem}{Theorem}
\newtheorem{corollary}{Corollary}
\newtheorem{proposition}{Proposition}
\newtheorem{problem}{Problem}

\begin{document}

\begin{frontmatter}


\title{On harmonic centers of graphs}



\author[TM]{Tomáš Madaras\corref{TMc}}
\ead{tomas.madaras@upjs.sk}
\cortext[TMc]{corresponding author}
\author[MP]{Matúš Paralič}
\ead{matus.paralic@student.upjs.sk}

\address[TM,MP]{P.J. Šafárik University in Košice, Slovakia}

\begin{abstract}
The harmonic centrality of a vertex $v$ in a graph $G = (V,E)$ is the sum of reciprocals of distances of vertices of $G$ from $v$. The vertices of $G$ which have the maximum (minimum) harmonic centrality form the harmonic center (or periphery, resp.) of $G$. We study harmonic centers of graphs and their localization in graph blocks, presenting sufficient conditions for graphs (in terms of diameter or number of edges) to have those centers contained in a single block; in addition,  we show that each connected graph is the harmonic center as well as harmonic periphery of some graphs.
\end{abstract}


\end{frontmatter}


\section{Introduction}
\label{S:1}

Throughout this paper, we consider (unless stated otherwise) connected graphs without loops or multiple edges. Given a graph $G = (V,E)$, the distance $d(u,v)$ of vertices $u,v \in V$ is the length of a shortest $u-v$-path. The {\em harmonic centrality} of a vertex $x$ of $G$ (\cite{MarchioriLatora}) is the sum $H_G(x) = \displaystyle\sum\limits_{{u \in V(G)} \atop {u \neq x}} \dfrac{1}{d(u,x)}$. The {\em harmonic center} ${\mathcal C}_H(G)$  of $G$ is the subgraph of $G$ induced by the vertices of $G$ that have the maximum harmonic centrality; conversely, the vertices of $G$ that have the minimum harmonic centrality induce in $G$ the {\em harmonic periphery} ${\mathcal P}_H(G)$. If ${\mathcal C}_H(G) \cong G$, then $G$ is called {\em harmonic-uniform}. We say that the harmonic center of $G$ is {\em delocalized} if there exist two distinct vertices of ${\mathcal C}_H(G)$ that belong to two different blocks of $G$.

The harmonic centrality is less-known member of the family of so called centrality indices, which are nonnegative valued functions on vertices of graphs that are invariant under graph isomorphism and reflect a subjective perception of central vertices in graphs (in the sense that vertices which are perceived as more central -- as the central vertex of a star compared to its pendant vertices -- have higher centrality values). In general, the closeness centrality and the betweenness centrality are considered as the most applied and studied centrality indices. Recall that the closeness centrality of a vertex $x$ in $G$ is the reciprocal $\dfrac{1}{\sum\limits_{u \in V(G)} d(u,x)}$ of the sum of all distances from $x$, and the betweenness centrality of $x$ is   the sum of relative numbers of shortest paths that pass through $x$ (formally, it is the sum $\displaystyle\sum\limits_{\{u,v\} \in {{V(G)\setminus \{x\}}\choose 2}} \dfrac{\sigma_{u,v}(x)}{\sigma_{u,v}}$ where $\sigma_{u,v}$ is the number of shortest $u-v$-paths in $G$, and $\sigma_{u,v}(x)$ is the number of those shortest $u-v$-paths that contain $x$ as an internal vertex). While harmonic centrality has also found various applications in social network analysis (for recent work, see, for example,  \cite{LiuWanZhang,YangMadsen,SartoriEtAl,GerlachBlumenthal,NiyazmandSheikholeslamKamali}),  its local graph-theoretic properties seem to be not yet explored in detail, except of the  connection to well-known Harary index, a global topological invariant defined as $\frac{1}{2}\displaystyle\sum\limits_{u \in V(G)} H_G(u)$ (see the monograph \cite{HararyIndex} for further results). 

In this paper, we focus on properties of harmonic centers resp. peripheries of graphs. We first show that every graph can appear as harmonic center (and, similarly, as harmonic periphery) of some graph, which aligns with analogous results for eccentricity-based center and periphery (\cite{Hedetniemi},\cite{BielakSyslo}), for median center (\cite{Slater}), detour center (\cite{ChartrandEtAl}) or betweenness center (see the recent work \cite{MadarasParalic}). On the other hand, the harmonic centrality surprisingly differs from other centrality indices regarding the centrality-uniform graphs (see \cite{HararyNorman, Truszczynski} and \cite{GagoHurajovaMadaras2}): we describe the construction of an infinite family of harmonic-uniform (nonregular) graphs with cut-vertices. Yet such graphs seem to be rare and somehow structurally constrained -- those of diameter at most 4 may contain at most one cut-vertex, and become 2-connected if their diameter is at most 3. The latter result follows from more general statement on localizations of harmonic center: for graphs of diameter at most 3, their harmonic center lies within a single block. However, for larger diameters, harmonic center may be  severely scattered even in trees: we describe the construction of infinite family of trees such that vertices of their harmonic centers are pairwise distant and do not lie on a single path; this is in sharp contrast with betweenness centers of trees (see \cite{HararyOstrand, MadarasParalic}: they are contained within a path) as well as with classical eccentricity-based or closeness centers (\cite{Jordan,Truszczynski}).

\section{Results}

We start our investigation by presenting results of realizability of graphs as harmonic centers and peripheries.

\begin{theorem}\label{thm:every-graph-harmonic-center}
For every $n$-vertex graph $G$, there exists a connected graph $F$ with $|V(F)| \leq 3n^2+3n+6$ such that ${\mathcal{C}_H(F)}\cong G$.
\end{theorem}

\begin{proof}

We describe, in several steps, the construction of $F$ from $G = (V,E)$ by gradually adding new auxiliary vertices; for the graph constructed in particular step, we calculate, in detail, harmonic centralities of its vertices. 

\medskip\noindent
{\bf Step 1. } Let $F_0$ be the join $G \vee K_1$ of $G$ with a new auxiliary vertex $x$. Clearly, $F_0$ is connected and its diameter is at most 2; hence, for a vertex $v \in V$,

$$
H_{F_0}(v) = 1 + \deg_G(v) + \frac{n-1-\deg_G(v)}{2} = \frac{n+1+\deg_G(v)}{2}.
$$
Note that $H_{F_0}(v)$ has integer value if and only if $n-1-\deg_G(v)$ is even.

\medskip\noindent
{\bf Step 2. } Let $P = \left\{ v\in V: n-1-\deg_G(v)\equiv 1\pmod 2 \right\}$, $p = |P|$.
For every vertex $v\in P$, add two new vertices  $a_v,b_v$
and new edges $va_v, a_vb_v, xa_v,$ $ xb_v$. Denote the resulting graph by $F_1$; note that, in this step, at most $2n$ auxiliary vertices were added.

For a fixed vertex $v\in P$, $d_{F_1}(v,a_v)=1,
d_{F_1}(v,b_v)=2$, hence $a_v,b_v$ contribute $1 + \frac{1}{2} = \frac{3}{2}$ to $H_{F_1}(v)$. For a vertex $w\in V\setminus\{v\}$, we have $d_{F_1}(w,a_v)=d_{F_1}(w,b_v)=2$
(there exist paths $w x a_v$ and $w x b_v$), so $a_v,b_v$ contribute $\frac12+\frac12=1$ to $H_{F_1}(w)$.

Observe that, when constructing $F_1$, the addition of new vertices does not change the distance between pairs of vertices from $V \cup \{x\}$. Hence, we conclude that, for $w \not \in P$, $H_{F_1}(w) = H_{F_0}(w) + p$, while, for 
$w\in P$, $
H_{F_1}(w)
=
H_{F_0}(w)+\frac32+(p-1)
=
H_{F_0}(w)+p+\frac12.
$. Also, $H_{F_0}(w) \in \mathbb{Z}$ for $w \not\in P$ and $H_{F_0}(w) \in \mathbb{Q}\setminus\mathbb{Z}$ for $w \in P$; thus, $H_{F_1}(v)$ is integer for every $v \in V$. Note also that, in $F_1$, the difference of maximum and minimum harmonic centrality of vertices from $V$ is at most $ \left \lfloor \frac{n-1}{2} \right \rfloor = R$.

\medskip\noindent
{\bf Step 3.} Now, we aim to construct a sequence $Q_0, Q_1, \dots, Q_t$ of graphs such that 
\begin{itemize}
    \item $Q_0 = F_1$, 
    \item for every $i = 0,\dots, t$, the values of $H_{Q_i}(v)$ are integers, and
    \item for every $v \in V$, $H_{Q_t}(v) = C$.
\end{itemize}

Assume that, for some $i = 0,\dots, t-1$, there are two vertices of $Q_i$ with different harmonic centralities. Put
$M_i = \max\limits_{u \in V} H_{Q_i}(u)$, find a vertex $v_i \in V$ such that $H_{Q_i}(v_i) = \min\limits_{u \in V} H_{Q_i}(u) = m_i$ and set $\Delta_i = M_i - m_i$ (by integrality of values of harmonic centralities of vertices from $V$ in $Q_i$, $\Delta_i$ are integers); note that $\Delta_i \leq R$. The graph $Q_{i+1}$ is constructed in the following way: take the set $Y_i = \{y^i_1,\dots, y^i_{2\Delta_i}\}$ of $2\Delta_i$ new vertices, and connect each of them to $x$ and $v_i$.

For a vertex $y\in Y_i$, $d_{Q_{i+1}}(v_i,y)=1$, and for $
w\in V\setminus\{v_i\}$, $d_{Q_{i+1}}(w,y)=2$ (due to the path $wxy$). Note also that the addition of vertices of $Y_i$ to $Q_i$ preserves the distances between pairs of other vertices. Therefore,
$H_{Q_{i+1}}(v_i)
=
H_{Q_i}(v_i)+2\Delta_i
=
H_{Q_i}(v_i)
+2\bigl(M_i-m_i\bigr)=
M_i+\Delta_i$.

Let $u \in V$ be a vertex of $Q_i$ with maximum harmonic centrality. Then $H_{Q_{i+1}}(u) = M_i + \Delta_i$ and, for $w \in V \setminus \{v_i\}$, $H_{Q_{i+1}}(w) = H_{Q_i}(w) + \Delta_i \leq M_i + \Delta_i$. So, in $Q_{i+1}$, $v_i$ has the maximum harmonic centrality while all vertices that had maximum harmonic centrality in $Q_i$ have maximum harmonic centrality also in $Q_{i+1}$.

At each step, the number of vertices of $V$ attaining the maximum harmonic centrality strictly increases. Hence, after at most $n-1$ steps, all vertices of $V$ have the same harmonic centrality. Therefore, there exists an integer $j \in \{0,\dots,n-1\}$ and a positive integer $C$ such that $F^* = Q_j$ and $H_{F^*}(v) = C$ for every vertex $v \in V$. Observe also that the number of auxiliary vertices added in Step 3 is at most $2(n-1)R = 2(n-1)\left \lfloor \frac{n-1}{2} \right \rfloor \leq (n-1)^2$.

\medskip\noindent
{\bf Step 4. } Put $k = |V(F^*) \setminus V|$ and $m = 2(n+k)$. Note that $k \geq 1$ (since $x \not \in V$) and $k \leq 1 + 2n + (n-1)^2 = n^2 +2$.

 Let $F$ be the graph obtained from $F^*$ as follows: take a set $S = \{s_1,\dots,$ $s_m\}$ of $m$ new vertices, and join every vertex of $S$ to every vertex of $V$. Since, in this step, $2(n+k)$ auxiliary vertices were added, we get that 

 $$
|V(F)| = (n+k) + 2(n+k) = 3(n+k) \leq 3(n+n^2+2) = 3n^2+3n+6.
 $$

Observe that the addition of $S$ to $F^*$ does not change the distance between vertices of $V$ in $F^*$. We then obtain

$$H_F(v) = H_{F^*}(v)+m = C+m \geq m.$$

We finish the proof by showing that the harmonic centrality of any vertex outside of $V$ is strictly less than $m$; since, in steps 1-4, no edges were added to join the vertices of $V$, $F[V] \cong {\mathcal C}_H(F)$.

Let $w \in V(F^*) \setminus V$. All other $n+k-1$ vertices of $F^*$ contribute at most 1 to the harmonic centrality of $w$. Since no vertex of $V(F^*) \setminus V$ is adjacent to a vertex of $S$, the distance of $w$ from each member of $S$ is at least 2. Hence, the contribution from $S$ is at most $m\cdot \frac{1}{2} = n+k$. Therefore,
$H_F(w) \leq (n+k-1) + (n+k) = 2(n+k) -1 = m-1$.

Let $s \in S$. In $F$, every vertex from $V$ is adjacent to $s$, so they contribute, in total, $n$ to the harmonic centrality of $s$. Every other vertex from $S$ has distance 2 from $s$, and their contribution is $\frac{m-1}{2}$. None of $k$ auxiliary vertices is adjacent to $s$, so their contribution is at most $k \cdot \frac{1}{2}$. Therefore,

$$
H_F(s) \leq n + \frac{m-1}{2} + \frac{k}{2} = n+\frac{2(n+k)-1}{2}+\frac{k}{2} =$$
$$2n + \frac{3k}{2} - \frac{1}{2} \leq 2n+2k-1 = m-1.
$$

\end{proof}


\begin{theorem}\label{thm:every-graph-harmonic-periphery}
For every $n$-vertex graph $G=(V,E)$, there exists a connected graph $F$ with $|V(F)| \leq 2n^2+3n+4$ such that ${\mathcal{P}_H(F)}\cong G$.
\end{theorem}

\begin{proof}
Construct the graph $F^*$ using the first three steps from the proof of Theorem \ref{thm:every-graph-harmonic-center}.  Recall that $G$ is the induced subgraph of $F^*$ and there exists $C$ such that, for every vertex $v \in V$, $H_{F^*}(v)=C$. In addition, every auxiliary vertex of $F^*$ distinct from $x$ is adjacent to $x$, and the number of vertices of $F^*$ is at most $n + n^2 + 2$. Taking into account  the fact that the maximum harmonic centrality in the graph  $Q_0 = F_1$ is at most $2n$, and, during the construction of the sequence $Q_0, Q_1, \dots, Q_t$, the actual increase of maximum harmonic centrality is at most $R$, we obtain the estimate

$$
C \leq 2n + (n-1)\left\lfloor\frac{n-1}{2} \right \rfloor \leq 2n + \frac{(n-1)^2}{2} = \frac{(n+1)^2}{2}.
$$

Let $W=V(F^*)\setminus V$, $k=|W|$ ($k \geq 1$ since $x \in W$). Put $m=2C+1$ and construct the graph $F$ taking the complete graph $K_m$ on the set $S=\{s_1,\dots,s_m\}$ and connecting each vertex from $S$ to each vertex of $W$. Observe that $|V(F)| = |V(F^*)|+m = |V(F^*)| + 2C+1 \leq (n^2+n+2) + (n^2 + 2n + 2) = 2n^2 + 3n + 4.$

Consider first the distance between $v \in V$ and $u \in V(F^*)$. Each $u-v$-path $P$ that visits $S$ enters $S$ from a vertex $w \in W$ and exits $S$ followed with a vertex $w' \in W$ (possibly, $w$ or $w'$ can be equal to $x$). The part of $P$ between $w$ and $w'$ has length at least 2, and can  be replaced by the subpath of length at most 2 through  $x$, yielding a $u-v$-path in $F^*$ which is not longer than $P$. Therefore, $d_F(v,u) = d_{F^*}(v,u)$. Next, the vertices $v \in V$ and $s \in S$ are not adjacent, but there exists the path $vxs$, so $d_F(v,s) = 2$. We thus obtain that, for $v \in V$,

$$
H_F(v)
=H_{F^*}(v)+\frac{m}{2}
=C+\frac{2C+1}{2}
=2C+\frac12.
$$

For $w\in W$,
 
$$
H_F(w)\geq m=2C+1>2C+\frac12,
$$

and for $s\in S$,

$$
H_F(s)
=(m-1)+k+\frac{n}{2}
=2C+k+\frac{n}{2}>2C+\frac12.
$$

We conclude that the vertices of $V$ have the same harmonic centrality which is smaller than harmonic centrality of any other vertex of $F$; hence $V$ induces the harmonic periphery in $F$.

\end{proof}


It is possible that Theorem \ref{thm:every-graph-harmonic-center} might be strengthened in a way that $F$ is highly-connected or, perhaps, even regular. Nevertheless, the graphs of specific graph families may have more constrained harmonic centers or peripheries. For example, it is easy to show that, in every tree of order at least 3, every pendant vertex has smaller harmonic centrality than its neighbor; hence, the harmonic center of a tree is contained in leaf-pruned part (the "trunk").
On the other hand, there are trees whose periphery contains none of its pendant vertices: take a 9-vertex path $v_1\dots v_5\dots v_9$, and add 37 pendant vertices to both $v_1$ and $v_9$. It follows by direct calculation that the minimum harmonic centrality is $\frac{569}{30} \doteq 18.9667$ and is attained at $v_5$, while every pendant vertex has harmonic centrality equal to $\frac{61813}{2520} \doteq 24.529$. 


\medskip
We turn now our attention to the properties of harmonic centers vs. block structure of graphs, and to harmonic-uniform graphs. We first show that, in general, harmonic centers can be severely delocalized even in trees:

\begin{theorem}
For all integers $k\geq 3, d \geq 1$, there exists a tree $T_{k,d}$ such that ${\mathcal C}_H(T_{k,d})$ is not contained within a path of $T_{k,d}$, and for any two vertices $x,y \in {\mathcal C}_H(T_{k,d})$, $d(x,y) \geq d$. 
\end{theorem}
\begin{proof}

Set
\[
    r=\max\left\{
        1,\left\lceil\frac{d}{2}\right\rceil,
        \left\lceil\frac{k-1}{2}\right\rceil
    \right\}
    \qquad\text{and}\qquad
    M=kr(r+1)(2r+1)+1.
\]
Then, $2r\ge d$ and $k-1\le 2r$.

The tree $T_{k,d}$ is constructed as follows: 
take $k$ paths $P_1,\dots, P_k$ of length $r$ and endvertices $v_i,u_i$ for $i = 1,\dots, k$. Identify all vertices $v_i$  (denote the resulting vertex as $c$ and the obtained subdivided $k$-star as $S$) and, for every $i = 1,\dots, k$ add $M$ new pendant vertices to $u_i$.

By symmetry, the harmonic centralities of $u_1,\dots, u_k$ are equal. Let
$x\in V(S)\setminus\{u_1,\ldots,u_k\}$
and $s=d(c,x)$ (note that $0\le s<r$). If $s >0$, then take a (unique) path $P_i$ containing $x$ in $S$; for $s=0$ (thus, $x=c$), take any of these paths.

Each of $M$ pendant neighbors of $u_i$ has distance $r-s+1$ from $x$, and every of the remaining  $(k-1)M$ pendant vertices of $T_{k,d}$ has distance $r+s+1$ from $x$.
Since pendant vertices of $T_{k,d}$ have no influence on distances between vertices of $S$, we obtain that
\[
    H_T(x)=H_S(x)+MF(s),
    \qquad
    F(s)=\frac{1}{r-s+1}+\frac{k-1}{r+s+1}.
\]
Similarly, 
\[
    H_T(u_i)=H_S(u_i)+MF(r).
\]

For $0\le s<r$,
\begin{align*}
    F(r)-F(s)
    &=
    1+\frac{k-1}{2r+1}
    -\frac{1}{r-s+1}
    -\frac{k-1}{r+s+1}\\
    &=
    (r-s)\left(
        \frac{1}{r-s+1}
        -\frac{k-1}{(2r+1)(r+s+1)}
    \right).
\end{align*}
Since $k-1\le 2r$, $r-s+1\le r+1$ and $r+s+1\ge r+1$,
we get
\begin{align*}
    F(r)-F(s)
    &\ge
    (r-s)\left(
        \frac{1}{r+1}
        -\frac{2r}{(2r+1)(r+1)}
    \right)\\
    &=\frac{r-s}{(r+1)(2r+1)}\\
    &\ge\frac{1}{(r+1)(2r+1)}.
\end{align*}

The tree $S$ has $kr+1$ vertices and the summands in the sum of harmonic centrality of a vertex are upper bounded by 1. So, for every $v\in V(S)$, $
    0\le H_S(v)\le kr.
$.
This implies (together with the choice of $M$)
\begin{align*}
    H_T(u_i)-H_T(x)
    &=
    H_S(u_i)-H_S(x)+M\bigl(F(r)-F(s)\bigr)\\
    &\ge
    -kr+\frac{M}{(r+1)(2r+1)}\\
    &=
    \frac{1}{(r+1)(2r+1)}>0.
\end{align*}
Hence, every vertex of $S$ which is distinct from $u_1,\ldots,u_k$
has strictly smaller harmonic centrality than $u_1,\dots, u_k$.

Now, let $w$ be a pendant vertex adjacent with $u_i$.
For every $z\in V(T)\setminus\{u_i,w\}$, we have
$d_T(w,z)=d_T(u_i,z)+1$. Thus
\[
    H_T(u_i)-H_T(w)
    =
    \sum_{z\in V(T)\setminus\{u_i,w\}}
    \left(
        \frac{1}{d_T(u_i,z)}
        -\frac{1}{d_T(u_i,z)+1}
    \right)>0
\]
implying that none of pendant vertices belongs to the harmonic center of $T_{k,d}$. Therefore, in total, ${\mathcal C}_H(T_{k,d}) =
    \{u_1,\ldots,u_k\}$. Furthermore, $
    d_T(u_i,u_j)
    =d_T(u_i,c)+d_T(c,u_j)
    =2r\ge d
$ and  $u_1,\ldots,u_k$ belong to distinct components of $T_{k,d}-c$, so, they do not lie on a single path.
\end{proof}


On the other hand, we obtain a positive localization result in the case of small diameter:

\begin{theorem}
\label{harmon_diam_3} Let $G$ be a graph of diameter at most 3. Then its harmonic center is contained in a single block.
\end{theorem}

\begin{proof} If $G$ consists of a single block, there is nothing to prove. Hence, assume that $G$ has a cut-vertex.Let $n = |V(G)|$ and $N_i(u)$ be the set of vertices of $G$ that have distance $i$ from $u$. If $\operatorname{diam}(G) = 2$ and $G$ is not 2-connected, then $G$ has unique cut-vertex $x$ (otherwise  $G$ contains a cut-vertex $y \neq x$ and, consequently, a path of length at least $1+d(x,y) + 1 \geq 3 $, a contradiction). It also holds that, for every $ v \in V(G)\setminus \{x\}$, $vx \in E(G)$; so $H(x) = n - 1$. For $ v \in V(G) \setminus \{x\}$, we have $  H(v) = \deg(v)\cdot 1 + (n - 1 - \deg(v)) \cdot \frac{1}{2} = \frac{\deg(v)}{2} + \frac{n-1}{2}$. Since $v$ can be adjacent only to vertices from its own block, $\deg(v) < n-1$, which yields $H(v) = \frac{\deg(v)}{2} + \frac{n-1}{2} < \frac{n-1}{2} + \frac{n-1}{2} = n-1 = H(x)$.

If $\operatorname{diam}(G) = 3$ and $G$ is not 2-connected, then all cut-vertices of $G$ induce a clique (otherwise $G$ contains cut-vertices $x\neq y$ and a path of length at least $1 + d(x,y) + 1 \geq 1+2+1 = 4$, a contradiction), hence, they belong to the same block. Further, if $G$ has more than one cut-vertex and $B$ is an end-block of $G$ containing a cut-vertex $x$, then all vertices of $B$ are adjacent to $x$. 

For a vertex $v$ of $G$,
 $$ 
 H(v) = \deg(v) + \frac{1}{2} \cdot |N_2(v)| + \frac{1}{3} \cdot |N_3(v)| =$$
 $$\deg(v) + \frac{1}{2} \cdot |N_2(v)| + \frac{1}{3} \cdot (n - 1 - \deg(v) - |N_2(v)|) =$$ $$  \frac{n - 1}{3} + \frac{2}{3} \cdot \deg(v) + \frac{1}{6} \cdot |N_2(v)|.
 $$ 
 
 Let $x$ be a cut-vertex of $G$; from the argument above, it follows that \(N_3(x) = \emptyset\), thus $$H(x) =  \frac{n - 1}{3} + \frac{2}{3} \cdot \deg(x) + \frac{1}{6} \cdot |N_2(x)| =$$ $$ \frac{n - 1}{3} + \frac{2}{3} \cdot \deg(x) + \frac{1}{6} \cdot (n - 1 - \deg(x)) =$$ $$ \frac{n - 1}{2} + \frac{\deg(x)}{2}$$. 
 
Let $B$ be an end-block of $G$ containing a cut-vertex $x$ such that every vertex of $B\setminus\{x\}$ is adjacent to $x$, and let $u \neq x$ be any vertex of $B$. Then
$$
H(u) = \frac{n - 1}{3} + \frac{2}{3} \cdot \deg(u) + \frac{1}{6} \cdot |N_2(u)| =
$$
$$
\frac{n - 1}{3} + \frac{2}{3} \cdot \deg(u) + \frac{1}{6} \cdot (\deg(x) - \deg(u)) =
$$
$$
\frac{n - 1}{3} + \frac{2}{3} \cdot \deg(u) + \frac{1}{6} \cdot \deg(x) - \frac{1}{6} \cdot \deg(u) =
$$
$$
\frac{n - 1}{3} + \frac{\deg(u)}{2} + \frac{\deg(x)}{6}.
$$
Since $\deg(x) < n$, we have
$\displaystyle \frac{\deg(x)}{6} < \frac{n}{6} < \frac{n -1}{6} + \frac{1}{2}$.
In addition,
\[
\frac{\deg(x)}{2}
=
\frac{\deg(x)}{6} + \frac{\deg(x)}{3}
<
\frac{n-1}{6} + \frac{1}{2} + \frac{\deg(x)}{3}.
\]
Also, $\deg(u) \leq \deg(x)-1$, hence
\[
\frac{\deg(u)}{2}
\leq
\frac{\deg(x)}{2} - \frac{1}{2}
<
\frac{n-1}{6} + \frac{\deg(x)}{3}.
\]
This gives
$$
H(u)
=
\frac{n-1}{3} + \frac{\deg(x)}{6} + \frac{\deg(u)}{2}
<
\frac{n-1}{3} + \frac{\deg(x)}{6}
+ \frac{n-1}{6} + \frac{\deg(x)}{3}
=
$$
$$
\frac{n-1}{2} + \frac{\deg(x)}{2}
=
H(x).
$$

Suppose first that $G$ has more than one cut-vertex. As observed above, all cut-vertices belong to the same block, say $B_0$, and every other block is an end-block whose vertices are all adjacent to its cut-vertex. Hence, by the preceding argument, no vertex outside $B_0$ can have maximum harmonic centrality. Therefore, ${\mathcal C}_H(G)$ is contained in $B_0$.

It remains to consider the case when $G$ has a unique cut-vertex $x$. We claim that at most one block containing $x$ can contain a vertex at distance 2 from $x$. Indeed, if two distinct blocks contained vertices $u$ and $v$, respectively, with
$d(x,u)=d(x,v)=2$, then every $u-v$ path would contain $x$, hence
$d(u,v)=d(u,x)+d(x,v)=4,$
contrary to $\operatorname{diam}(G)=3$.

Now, if no such block exists, then every vertex distinct from $x$ belongs to an end-block whose vertices are all adjacent to $x$. By the preceding argument, $H(u)<H(x)$ for every $u\neq x$, and hence ${\mathcal C}_H(G)=\{x\}$.

Otherwise, let $B^*$ be the unique block containing a vertex at distance 2 from $x$. Every block distinct from $B^*$ has all its vertices adjacent to $x$, and therefore every vertex $u$ outside $B^*$ satisfies $H(u)<H(x)$. Since $x\in V(B^*)$, it follows that every vertex of maximum harmonic centrality belongs to $B^*$. Thus, also in this case, ${\mathcal C}_H(G)$ is contained in a single block.
 \end{proof}

\begin{corollary}\label{cor:uniform-diameter-three}
Every harmonic-uniform graph on at least three vertices with diameter at most 3 is 2-connected.
\end{corollary}

\begin{corollary} 
\label{harmon_many_edges} Let $G$ be an $n$-vertex graph of size at least $4+\frac{n(n-5)}{2}$. Then its harmonic center is contained in a single block.
\end{corollary}

\begin{proof} By \cite{Ore}, the maximum number of edges of an $n$-vertex graph of diameter  $d$ is \( d+ \frac{(n-d-1)  (n-d+4)}{2}\) (note that $n \geq d+1$). For $d=4$, we obtain that if $G$ has more than  \( 4 + \frac{n (n-5)}{2}\) edges, its diameter is at most 3; by Theorem \ref{harmon_diam_3}, the harmonic center of $G$ is then contained in a single block. The paper \cite{Ore} contains also the characterization of $n$-vertex graphs of diameter 4 with maximum number of edges (which is $4 + \frac{n (n-5)}{2}$): such a graph is either obtained from $K_{n-2}$ by deleting an edge $ab$ and adding a new path $avu$, or is formed as follows: for some integer $x$ with $1 \leq x \leq n-4$, take the join of $K_x$ and $K_{n-3-x}$, add a new vertex $z$ which is joined to all vertices of $K_{n-3-x}$, a new vertex $u$ which is joined to all vertices of $K_x$, and a new vertex $v$ which is joined only to $u$.

For the former graph, $H(u) = 1 + \frac{1}{2} + \frac{n-4}{3}+\frac{1}{4} = \frac{7}{4}+\frac{n-4}{3},H(v) = 2\cdot 1 + \frac{n-4}{2} + \frac{1}{3} = \frac{7}{3} + \frac{n-4}{2}, H(a)= n-3+2\cdot\frac{1}{2} = n-2$.
Hence, for  \( n \geq 5\), we have that  \(H(a) > H(v) > H(u)\), so the harmonic center is contained in the block which has $n-2$ vertices. For the latter graph the statement follows from the fact that $v$ is a pendant vertex (so \(H(v) < H(u)\)) and all other vertices form the second block. 
\end{proof}

To address the quality of the lower bound in Corollary \ref{harmon_many_edges}, we describe, for every \(n\geq7\), the construction of an $n$-vertex graph with \( \frac{n^2}{4} - O(n)\) edges, which has delocalized harmonic center. For odd $n = 2k+1$, take two copies of $K_k$, pick a vertex $u$ in one copy and a vertex $u'$ in the other, and add a new path $uvu'$. The obtained graph has $n$ vertices, $2{k \choose 2} + 2 = \frac{n-1}{2}\left (\frac{n-1}{2}-1 \right ) + 2 = \frac{n^2}{4}-n+\frac{11}{4}$ edges and three cut-vertices $u,v,u'$; by symmetry, $H(u) = H(u')$.

Moreover, if $w\neq u$ is a vertex of the copy of $K_k$ containing $u$, then $u$ and $w$ have the same total contribution to their harmonic centralities from the vertices of this copy, while every vertex outside this copy is farther from $w$ than from $u$. Hence, $H(w)<H(u)$; by symmetry, the analogous statement holds for the other copy of $K_k$. Thus, it remains to compare $H(u)$ with $H(v)$.

We have \(H(v) = 2 + \frac{1}{2}\cdot2(k-1) = k+1\) and \(H(u) = k + \frac{1}{2} + \frac{1}{3}\cdot(k-1) = \frac{4}{3} k + \frac{1}{6}\). From $n \geq 7$, we have $k \geq 3$, which implies \(\frac{4}{3} k + \frac{1}{6} \geq k+1\), yielding $H(u) \geq H(v)$. For $n$ even, connect vertices $u,u'$ from distinct copies of $K_k$ by a new path $uvv'u'$. As above, every vertex of either copy of $K_k$ distinct from $u$ or $u'$, respectively, has smaller harmonic centrality than $u$ or $u'$. By symmetry, $H(u)=H(u')$ and $H(v)=H(v')$, so it is again enough to compare $H(u)$ with $H(v)$.
We have \(H(v) = 2 + \frac{1}{2}\cdot(k-1+1) + \frac{1}{3}\cdot(k-1) =  \frac{5}{6}k+ \frac{5}{3}\) and \(H(u) = k + \frac{1}{2} + \frac{1}{3} + \frac{1}{4}(k-1)=\frac{5}{4}k+\frac{7}{12}\); hence, $H(u) \geq H(v)$ if and only if  \(\frac{5}{4}k+\frac{7}{12} \geq \frac{5}{6}k+ \frac{5}{3}\), which holds for  \(k \geq 3\). The obtained graph has $n = 2k+2 \geq 8$ vertices and \(2{k \choose 2} + 3 = \frac{n-2}{2}\cdot \left (\frac{n-2}{2}-1\right ) + 3 = \frac{n^2}{4} - \frac{3n}{2} + 5\) edges.

\bigskip
 The Figure 1 shows all 168 graphs on 5--10 vertices which are harmonic uniform. They were obtained by exhaustive computer search over all non-isomorphic connected graphs whose collection is available at the graph data webpage  \texttt{https://users.cecs.anu.edu.au/\textasciitilde bdm/data/graphs.html} of \linebreak Brendan McKay; the harmonic centralities of vertices in these graphs were calculated using {\it HarmonicCentrality(..)} procedure of Maple 2020 computer algebra system. All of them are regular, 2-connected and almost all (except $C_8,C_9$ and $C_{10}$) have diameter at most 3. In fact, small diameter has further non-trivial implications on the structure of these graphs:

\begin{figure}
\includegraphics[width=5.4in]{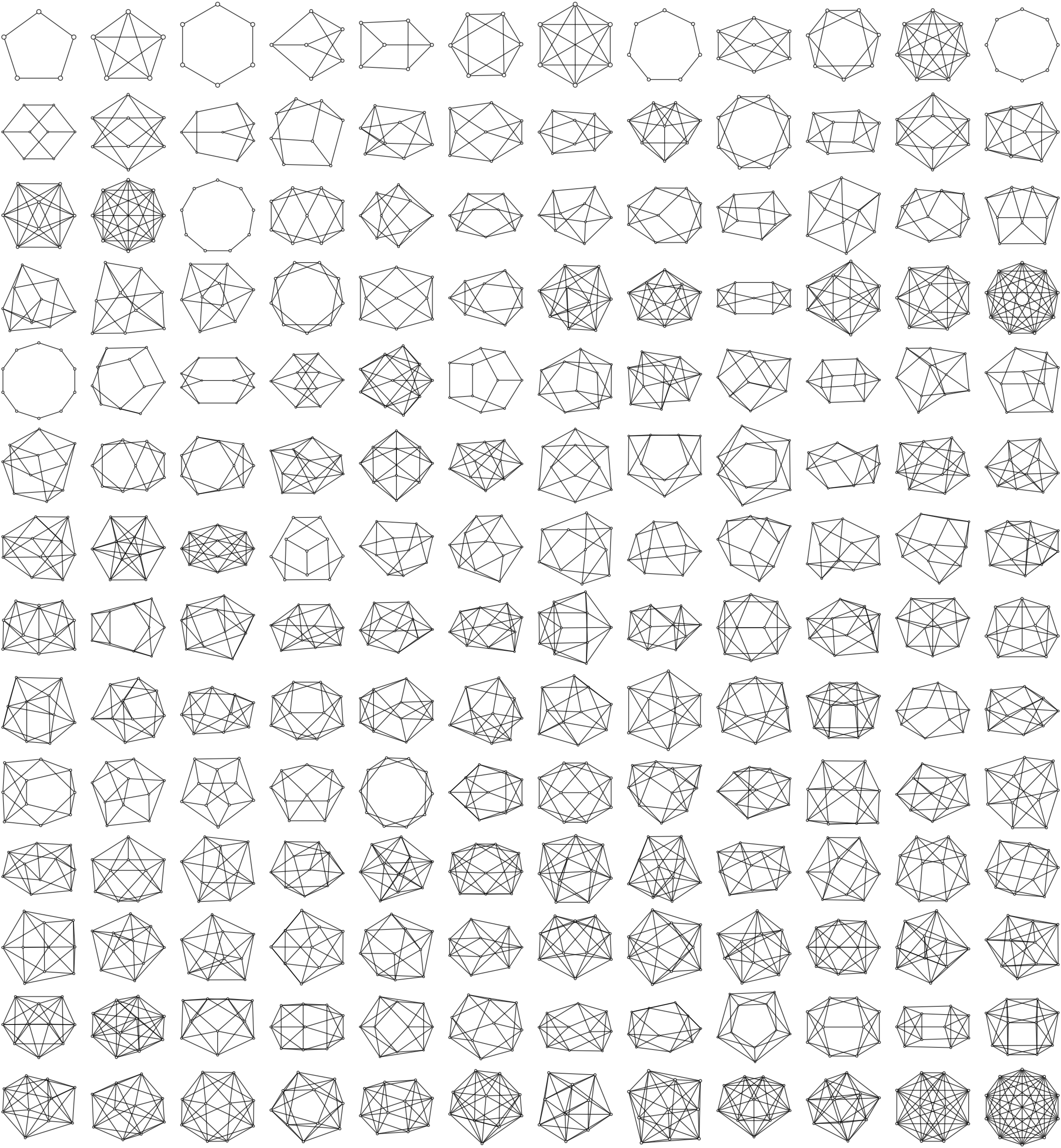}
\caption{All harmonic-uniform graphs on 5--10 vertices}
\end{figure}

\begin{theorem}\label{thm:uniform-at-most-one-cutvertex}
Let $G$ be a harmonic-uniform graph with $\operatorname{diam}(G)\leq4$.
Then $G$ contains at most one cut-vertex.
\end{theorem}

\begin{proof}
By contradiction. Let $G$ be a harmonic-uniform graph of diameter at most 4 that contains at least two cut-vertices; we can select two distinct cut-vertices $c_1,c_2$ from two end-blocks $B_1,B_2$ of $G$. For $i \in \{1,2\}$, set 
$r_i=\max\left\{ d_G(c_i,v):\ v\in V(B_i)\setminus\{c_i\} \right\}$
and take the vertex $u_i\in V(B_i)\setminus\{c_i\}$ such that
$d_G(c_i,u_i)=r_i$. Every path from $u_1$ to $u_2$ contains both $c_1$ and $c_2$, thus,
$r_1 + r_2 + 1 \leq
r_1+d_G(c_1,c_2)+r_2 = d_G(u_1,u_2) \leq 4$. This implies that one of $r_1, r_2$ -- say, $r_1$ -- equals 1, and so $c_1$ is adjacent to every other vertex from $V(B_1)$. 
Taking any vertex $u \in V(B_1) \setminus \{c_1\}$ and a vertex $z \in V(G) \setminus V(B_1)$ (note that since $G$ has a cut-vertex distinct from $c_1$, $z$ exists), every $u-z$-path contains $c_1$, so $d(u,z) = 1+d(c_1,z)$ and $\frac{1}{d(c_1,z)} > \frac{1}{d(u,z)}$. For $z' \in V(B_1) \setminus \{c_1,u\}$, $d(c_1,z') = 1 \leq d(u,z')$; this yields that the contribution of $z'$ to $H(c_1)$ is not smaller than to $H(u)$ while the contribution of $z$ to $H(c_1)$ is larger than the contribution of $z$ to $H(u)$.
Hence, 
$
H_G(c_1)>H_G(u)
$, a contradiction with harmonic uniformity of $G$.
\end{proof}

Before describing the construction of harmonic-uniform graphs of diameter 4 and  vertex connectivity 1, we state an auxiliary result on harmonic uniformity of specific graphs with single cut-vertex and two identical blocks:

\begin{proposition}\label{prop:general-two-layer-construction}
Let $x_1,x_2,k,\ell,m$ be positive integers, $R$ be a $k$-regular graph of order $x_1$ and $Q$ be a bipartite graph with parts $A,B$ such that $|A|=x_1,|B|=x_2$ and every vertex of $A$ (resp. $B$) has degree $\ell$ (resp. $m$). Let $F$ be a graph consisting of two identical blocks ${\mathcal B}_1,{\mathcal B}_2$ with common cut-vertex $c$ based on two copies $Q_1,Q_2$ of $Q$ such that, for $i \in \{1,2\}$,
\begin{itemize}
    \item $F[A_i] \cong R$,
    \item $F[B_i] \cong K_{x_2}$,
    \item $c$ is adjacent to every vertex of $A_i$.
\end{itemize}

\noindent
If 
\begin{align}
2x_1+x_2
&=
(1+k+\ell)
+\frac{x_1-k-1}{2}
+\frac{x_2-\ell}{2}
+\frac{x_1}{2}
+\frac{x_2}{3},
\label{eq:two-layer-A}\\
2x_1+x_2
&=
(x_2-1+m)
+\frac{x_1-m}{2}
+\frac12
+\frac{x_1}{3}
+\frac{x_2}{4},
\label{eq:two-layer-B}
\end{align}
then $F$ is harmonic-uniform.
\end{proposition}

\begin{proof}
All $2x_1$ vertices of $A_1\cup A_2$ are adjacent to $c$ while $2x_2$ vertices of $B_1\cup B_2$
have, from $c$, the distance $2$. Thus $
H_F(c)=2x_1+x_2$.

Let $a\in A_i$ and $j=3-i$. Then $c$ is adjacent to $a$ as well as $k$ neighbors of $a$ in $A_i$ and $\ell$ neighbors of $a$ in $B_i$. The remaining vertices of $A_i$ and $B_i$, and all vertices of $A_j$ have distance 2 from $a$, and all vertices of $B_j$ have distance 3 from $a$. Hence,

\[
H_F(a)
=
(1+k+\ell)
+\frac{x_1-k-1}{2}
+\frac{x_2-\ell}{2}
+\frac{x_1}{2}
+\frac{x_2}{3}.
\]

Similarly, for $b\in B_i$, the remaining $x_2-1$ vertices of $B_i$ are adjacent to $b$ as well as its $m$ neighbors in  $A_i$. The remaining neighbors of $A_i$ and the vertex $c$ have distance 2 from $b$, all vertices of $A_j$ have distance  $3$ from $b$ while the distance of all vertices of $B_j$
from $b$ is $4$. Therefore,
\[
H_F(b)
=
(x_2-1+m)
+\frac{x_1-m}{2}
+\frac12
+\frac{x_1}{3}
+\frac{x_2}{4}.
\]

The conditions \eqref{eq:two-layer-A} and \eqref{eq:two-layer-B} imply that
\[
H_F(c)=H_F(a)=H_F(b).
\]
So $F$ is harmonic-uniform.
\end{proof}

\begin{theorem}\label{thm:infinite-uniform-family}
For every integer $t\geq2$ there exists a non-regular  harmonic-uniform graph
$F_t$ with single cut-vertex such that $
|V(F_t)|=48t+1$ and $
\operatorname{diam}(F_t)=4$.
\end{theorem}

\begin{proof}
Fix $t\geq2$ and put
\[
x_1=6t,\qquad x_2=18t,
\]
\[
k=3t-4,\qquad \ell=15t+3,\qquad m=5t+1.
\]

First, consider a $(3t-4)$-regular circulant graph on $6t$ vertices. It may be constructed as follows: for $3t-4$ even, index its vertices by elements of
$\mathbb{Z}_{6t}$ and connect the vertex $j$ with vertices
$
j\pm1,\dots,j\pm\frac{3t-4}{2}.
$
; for  $3t-4$ odd, connect $j$ with
$
j\pm1,\dots,j\pm\frac{3t-5}{2}
$
and, in addition, with its antipodal vertex $j+3t$ (all indices are taken modulo
$6t$).
Let $A_i,i\in\{1,2\}$ be vertex sets of two copies of this circulant.

Let
\[
B_i=
\left\{
b_{r,q}:
r\in\{0,1,2\},\
q\in\mathbb{Z}_{6t}
\right\}.
\]
be a vertex set of a copy of the complete graph $K_{18t}$.
Now, connect a vertex $a_j\in A_i$ with $b_{r,q}\in B_i$ if and only if
$
q-j\in\{0,1,\dots,5t\}
\pmod{6t}.
$. Finally, let $F_t$ be obtained by connecting every vertex from $A_1 \cup A_2$ to a new vertex $c$.

Every vertex from $A_i$ has exactly $3(5t+1)=15t+3$ neighbors while every vertex from  $B_i$ has exactly $5t+1$ neighbors in $A_i$. Hence $F_t$ satisfies the conditions of Proposition  \ref{prop:general-two-layer-construction} and the harmonic centrality of each of its vertices equals $30t$. Furthermore, 
$
|V(F_t)|
=
1+2x_1+2x_2
=
48t+1
$ and degrees of $F_t$ are $12t$ (for $c$), $1+k+\ell = 18t$ (for vertices from $A_1 \cup A_2$)  and 
$x_2-1+m=23t$ (for $b\in B_1\cup B_2$), respectively. Also, no vertex from $A_1 \cup A_2$ and from $B_1 \cup B_2$ is a cut-vertex, and the vertices from these sets have distance at most $2$ from the cut-vertex $c$ (equality holds for vertices from $B_i$). So, $\operatorname{diam}(F_t)=4$.
\end{proof}

\bigskip
For $t=2$, we obtain that
$
x_1=12,
x_2=36,
k=2,
\ell=33,
m=11$; thus $F_2$ has $97$ vertices. 

\bigskip
By help of AI-driven approach, developing the above constructions with chatGPT, we present an infinite family $\{F'_t\}_{t=0}^\infty$ of harmonic-uniform graphs of diameter 5 that contain a unique cut-edge. Set $n = 30t+10,k = 22t + 5$ and $m=23t+10$. Then take a $k$-regular circulant graph on the set $A_i, i \in \{1,2\}$ of $n$ vertices (since $n$ is even and $k <n$, such a circulant graph exists), a complete graph $K_n$ on the vertex set $B_i$ and add new edges between $A_i$ and $B_i$ to form an $m$-regular bipartite graph with parts $A_i,B_i$ (again, this can be done since $m \leq n$). Finally, take a new vertex $c_i$, connect it with all vertices from $A_i$ and add the edge $c_1 c_2$.

The resulting graph $F'_t$ has diameter $d=5$ and $c_1 c_2$ is its unique cut-edge. Further, the distance profile (that is, $d$-tuple whose $j$-th entry is the number of vertices of distance $j$ from a selected vertex) of $c_i, a \in A_i$ and $b \in B_i$ is $(n+1,2n,n,0,0),(1+k+m,2n-k-m,n,n,0)$ and $(n-1+m,n-m+1,1,n,n)$, respectively. We then obtain

$$
H(c_i)=n+1 + \frac{1}{2} \cdot 2n + \frac{1}{3}\cdot n = 2n + 1 + \frac{n}{3} = 70t + \frac{73}{3},
$$

$$
H(a) = 1 +k+m + \frac{1}{2}\cdot (2n-k-m) + \frac{1}{3} \cdot n + \frac{1}{4}\cdot n = 70t + \frac{73}{3},
$$

$$
H(b) = n-1+m + \frac{1}{2}\cdot (n-m+1) + \frac{1}{3}\cdot 1 + \frac{1}{4}\cdot n + \frac{1}{5}\cdot n = 70t + \frac{73}{3}.
$$
Thus, $F'_t$ is harmonic-uniform. The graph $F'_0$ has 42 vertices; so far, it is the smallest known harmonic-uniform graph which is not 2-connected.

The above constructions are based on gluing two identical blocks, so the resulting graphs possess non-trivial automorphism. On the other hand, it is also possible to construct harmonic-uniform graphs (albeit not of vertex connectivity one) with arbitrary automorphism groups: by \cite{Phelps}, for any finite group $A$, there exists a strongly regular graph whose automorphism group is isomorphic to $A$. Since every connected non-complete strongly regular graph has diameter $2$, the distance profiles (and, consequently, harmonic centralities) of all its vertices are the same.

\section{Concluding remarks}

Using AI-driven approach by chatGPT, it is possible to construct also harmonic-uni\-form graphs with diameter greater than 6 that contain a single cut-edge; the fact that such graphs are very large, together with the observation that all known harmonic-uniform graphs have at most two cut-vertices, raises several open questions:

\begin{problem}
Given a block-graph $G$ (that is, a graph whose every block is a clique), does there exist a harmonic-uniform graph $G'$ such that the intersection graph of the blocks of $G'$ is isomorphic to $G$?
\end{problem}

\begin{problem}
Let $c(n)$ be the maximum number of cut-vertices in a harmonic-uniform graph on $n$ vertices. Determine $c(n)$, or at least its asymptotic order of growth. Is $c(n) = o(n)$?
\end{problem}

\medskip\noindent
{\bf Acknowledgement. } This research was supported by the Slovak Research and Development Agency under the Contract No. APVV-23-0191.






\bibliographystyle{elsarticle-num-names-alphsort}



\end{document}